\documentclass[12pt,leqno]{amsart}
\usepackage{color,amsmath,amsfonts,amssymb,amscd,amsthm,amsbsy,upref, siunitx}
\numberwithin{equation}{section}
\def\hangbox to #1 #2{\vskip3pt\hangindent #1\noindent \hbox to #1{#2}$\!\!$}

\newtheorem{thm}{Theorem}[section]

\newtheorem{lem}[thm]{Lemma}

\newtheorem{example}[thm]{Example}
\newtheorem{prop}[thm]{Proposition}
\theoremstyle{definition}

\newtheorem{defn}[thm]{Definition}

\newtheorem{problem}[thm]{Problem}

\theoremstyle{remark}

\def\N{{\mathbb N}}
\def\R{{\mathbb R}}

\def\Z{{\mathbb Z}}

\newcommand{\norm}[1]{\left\lVert#1\right\rVert}

\def\sfrac#1#2{\kern.1em\raise.5ex\hbox{$#1$}
        \kern-.1em/\kern-.05em\lower.25ex\hbox{$#2$}}

\def\vp{\varepsilon}

\newcommand{\fw}{\text{\fw}}

\begin{document}
\allowdisplaybreaks
\title{Continuous Schauder frames for Banach spaces}

\author{Joseph Eisner}
\address{Department of Mathematics\\
University of Virginia\\
Charlottesville, VA 22901  USA}
 \email{je5pd@virginia.edu }

\author{ Daniel Freeman}
\address{Department of Mathematics and Statistics\\
St Louis University\\
St Louis MO 63103  USA\\ and\\ Department of Mathematics\\ Duke University\\ Durham NC 27708 USA
} \email{daniel.freeman@slu.edu}

\thanks{
The second author was supported by grant 353293 from the Simon's foundation.  This paper forms a portion of the first authors masters thesis which was prepared at St Louis University.}

\thanks{2010 \textit{Mathematics Subject Classification}: 42C15, 81R30, 46B10 }

\begin{abstract}

We introduce the notion of a continuous Schauder frame for a Banach space.  This is both a generalization of continuous frames and coherent states for Hilbert spaces and a generalization of unconditional Schauder frames for Banach spaces.  As a natural example, we prove that any wavelet for $L_p(\R)$ with $1<p<\infty$ generates a continuous wavelet Schauder frame.  Furthermore, we generalize the properties shrinking and boundedly complete to the continuous Schauder frame setting, and prove that many of the fundamental James theorems still hold in this general context.

\end{abstract}

\maketitle

\section{Introduction}\label{S:Intro}
In Hilbert spaces, frames and orthonormal bases give discrete ways to represent vectors using series, and continuous frames and coherent states give continuous ways to represent vectors using integrals.    A {\em frame} for a Hilbert space $H$ is a collection of vectors
$(x_j)_{j\in J}\subset H$ for which there exists constants
$0\leq A\leq B$ such that for any $x\in H$,
$A\|x\|^2\leq \sum_{j\in J}|\langle x, x_j\rangle|^2\leq
B\|x\|^2$.
 Given any frame $(x_j)_{j\in J}$ for a Hilbert space $H$,
 there exists a frame $(f_j)_{j\in J}$ for $H$, called a {\em
dual frame}, such that 
\begin{equation}\label{E:0}
x=\sum_{j\in J} \langle x, f_j\rangle x_j\quad \textrm{ for all $x\in H$}.
\end{equation}
The equality in (\ref{E:0}) allows the reconstruction of any vector
$x$ in the Hilbert space from the sequence of coefficients $(\langle
x, f_j\rangle)_{j\in J}$.  Continuous frames and coherent states are a generalization of frames, in that instead of summing over a discrete set, we integrate over a measure space.   Coherent states were invented  by Schr\"odinger in 1926 \cite{S} and were generalized to continuous frames by Ali, Antoine, and Gazeau \cite{AAG1}. The short time Fourier transform and the continuous wavelet transform are two particularly important examples of continuous frames.  
Let $(M, \Sigma, \mu)$ be a $\sigma$-finite measure space and let $H$ be a separable Hilbert space.  A measurable function $\psi:M\rightarrow H$ is a {\em continuous frame} of $H$ with respect to $\mu$ if there exists constants $A,B>0$ such that for all $x\in H$, 
$A\|x\|^2\leq \int_M |\langle x, \psi(t)\rangle|^2 d\mu(t)\leq B \|x\|^2$.
If $A=B=1$ then the continuous frame is called a {\em coherent state}.
As is the case with frames, any continuous frame may be used to reconstruct vectors using a dual frame.  That is, if $\psi:M\rightarrow H$ is a continuous frame, then there exists a {\em dual frame} $\phi:M\rightarrow H$ such that
\begin{equation}\label{E:c0}
x=\int_M \langle x, \phi(t)\rangle \psi(t) d\mu(t)\quad \textrm{ for all $x\in H$}.
\end{equation}
Equation \eqref{E:c0} involves integrating vectors in a Hilbert space, and is defined weakly using the Pettis Integral.  We will define the Pettis Integral and discuss it further in Section \ref{S:CSF}.

Frames for Hilbert spaces have been generalized to Banach spaces in multiple ways, such as atomic decompositions \cite{FG1}, Banach frames \cite{G}, framings \cite{CHL}, and Schauder frames \cite{CDOSZ}.  These particular methods are based on extending the reconstruction formula \eqref{E:0} to Banach spaces.  Given a Banach space $X$ with dual $X^*$, a sequence of pairs $(x_j,f_j)_{j=1}^\infty\subseteq X\times X^*$ is called a {\em Schauder frame} of $X$ if
\begin{equation}\label{E:SF}
x=\sum_{j=1}^\infty f_j(x) x_j\quad \textrm{ for all $x\in X$}.
\end{equation}
Thus, Schauder frames are direct generalizations of the reconstruction formula for frames in Hilbert spaces.  A Schauder frame is called {\em unconditional} or a {\em framing} if the series in \eqref{E:SF} converges in any order.  Though coherent states and continuous frames in Hilbert spaces have long been studied and play important roles in mathematical physics and harmonic analysis, the natural extension to continuous Schauder frames has only been considered for certain Banach spaces using coorbit theory \cite{FG1}\cite{FG2}\cite{FG3}\cite{FoR} and for complemented subspaces of $L_p$ using p-frames \cite{FO}.  
Given a Banach space $X$ with dual $X^*$ and a $\sigma$-finite measure space $(M, \Sigma, \mu)$, we call  a measurable function $t\mapsto(x_t,f_t)\in X\times X^*$ a {\em continuous Schauder frame} of $X$ if for all $x\in X$,
\begin{equation}\label{E:cSFdef}
x=\int_M f_t(x) x_t d\mu(t).
\end{equation}
As with continuous frames for Hilbert spaces, the integral in Equation \eqref{E:cSFdef} involves integrating vectors and is defined weakly using the Pettis Integral which we define in Section \ref{S:CSF}.  Unlike series, there is no order for integration, and so all continuous Schauder frames are by necessity unconditional.  
 In the case that the measure space $(M,\mu)$ is simply the natural numbers with counting measure, then $(x_n, f_n)_{n\in\N}$ is a continuous Schauder frame if and only if it is an unconditional Schauder frame.   Thus, continuous frames are indeed generalizations of unconditional Schauder frames.

 Many coherent states and continuous frames are of the form $(S^a f)_{a\in \Lambda}$ or $(S^aT^b f)_{(a,b)\in \Lambda}$ for some isometries $S$ and $T$ and continuous ring $\Lambda$.  Many important bases and frames can be obtained by sampling the continuous frame at a discrete sub-lattice of $\Lambda$.  For example, wavelet bases and Gabor frames can be obtained by sampling the continuous wavelet transform and short time Fourier transform respectively.  Discrete frames and bases are well suited for computations, but the continuous frames they are sampled from are often very useful to work with directly themselves.  For example, if one translates a function in $L_2(\R)$ then the resulting continuous wavelet frame coefficients are simply shifted by the amount translated, however if one translates a function by a non-integer value then the resulting discrete wavelet coefficients are completely different.  In Section \ref{S:ex} we show that any discrete wavelet for $L_p(\R)$ for $1<p<\infty$ gives rise to a continuous wavelet Schauder frame and thus much of the analysis that is done for the continuous wavelet transform in Hilbert spaces may be done in the Banach space setting as well.
 
 We also consider how to generalize other properties of Schauder frames to continuous Schauder frames.  In \cite{CL} and \cite{L}, they define the properties shrinking and boundedly complete for Schauder frames and prove that many of James' classic theorems \cite{Ja} on shrinking and boundedly complete Schauder bases can be extended to Schauder frames.  In particular, a Schauder frame $(x_j,f_j)_{j=1}^\infty$ of  $X$ is shrinking if and only if $(f_j,x_j)_{j=1}^\infty$ is a Schauder frame of $X^*$, and if $(x_j,f_j)_{j=1}^\infty$ is shrinking and boundedly complete then $X$ is reflexive.  On the other hand, in \cite{BFL} they prove that every infinite dimensional Banach space which has a Schauder frame also has a non-shrinking Schauder frame, so unlike for Schauder bases the converse of the previous theorem for Schuader frames does not hold.  In \cite{CLS}, they prove that an unconditional Schauder frame is shrinking if and only if the Banach space does not contain $\ell_1$, and that an unconditional Schauder frame is boundedly complete if and only if the Banach space does not contain $c_0$.  In Section \ref{S:sbc} we define what it means for a continuous Schauder frame to be shrinking or boundedly complete, and we prove that the previous stated theorems are true for continuous Schauder frames as well.  Sections \ref{S:ex} and \ref{S:ex2} are devoted to constructing examples of continuous Schauder frames in $L_p(\R)$ and $\ell_p$ for $1<p<\infty$.  Our final section proposes an open problem on when continuous Schauder frames may be sampled to obtain a discrete Schauder frame.

 \section{The Pettis Integral and Continuous Schauder frames}\label{S:CSF}
 The main concept of the Pettis integral is to integrate vector valued functions by considering the Lebesgue integrals of the real valued functions formed by composing with linear functionals.  This method allows one to transfer many of the fundamental properties of Lebesgue integration to the Banach space setting.  
\begin{defn} Let $(M, \Sigma, \mu)$ be a $\sigma$-finite measure space and $X$ a separable Banach Space. A vector-valued function $F: M \rightarrow X$ is said to be $\mu$-\textit{Pettis integrable} (or \textit{Pettis integrable}, or merely \textit{integrable} if context is understood) if for any $E \in \Sigma$ there exists $x_E \in X$ such that $f(x_E) = \int_E f(F) d\mu$ for all $f \in X^{*}$ (where this latter integral is Lebesgue). Then we say $\int_E F d\mu = x_E$ and, in particular, $\int F d\mu = x_M$.
\end{defn}
If the vector valued map takes values in a dual space $X^*$ then one can instead consider just using the weak*-continuous linear functionals.
\begin{defn} Let $(M, \Sigma, \mu)$ be a $\sigma$-finite measure space and $X$ a separable Banach Space with dual $X^{*}$. A functional-valued function $G: M \rightarrow X^{*}$ is said to be $\mu$-\textit{Pettis* integrable} (or \textit{Pettis *integrable}, or merely \textit{*integrable} if context is understood) if for any $E \in \Sigma$ there exists $f_E \in X^{*}$ such that $f_E(x) = \int_E G(x) d\mu$ for all $x \in X$. Then we say $\int_E^{*} G d\mu = f_E$ and, in particular, $\int^{*} G d\mu = f_M$.
\end{defn}
 
 Recall that we use the Pettis integral to define continuous Schauder frames, and we will use the Pettis* integral to define continuous* Schauder frames.
 \begin{defn}
 Given a separable Banach space $X$ with dual $X^*$ and a $\sigma$-finite measure space $(M, \Sigma, \mu)$,  a measurable function $t\mapsto(x_t,f_t)\in X\times X^*$ is called a {\em continuous Schauder frame} of $X$ if for all $x\in X$,
\begin{equation}\label{E:cSF}
x=\int_M f_t(x) x_t d\mu(t).
\end{equation}
 The dual map $t\mapsto(f_t,x_t)\in X^*\times X$ is called a {\em continuous* Schauder frame} of $X$ if for all $f\in X$,
\begin{equation}\label{E:cSF}
f=\int^*_M f(x_t) f_t d\mu(t).
\end{equation}
 \end{defn}

 A sequence of vectors $(x_j)_{j=1}^\infty$ is a Schauder basis for a Banach space $X$, if and only if the biorthogonal functionals $(x^*_j)_{j=1}^\infty$ are a $w^*$-Schauder basis for $X^*$.  That is, for all $f\in X^*$, the series $\sum_{j=1}^\infty f(x_j) x^*_j$ converges $w^*$ to $f$.  
We will prove in Lemma \ref{L:DualFrame1} that if $(x_t,f_t)_{t\in M}$ is a continuous Schauder frame of $X$ then $(f_t,x_t)_{t\in M}$ is a continuous* Schauder frame of $X^*$.  This relationship will be useful for us when using duality techniques.  However, Example \ref{E:notCSF} shows that the converse does not always hold for continuous Schauder frames.

By the definition of the Pettis integral, $t\mapsto(x_t,f_t)\in X\times X^*$ is a continuous Schauder frame of $X$ if and only if for all $x\in X$, $f\in X^*$, and $E\in \Sigma$ there exists $x_E\in X$ such that $x=x_X$ and
\begin{equation}\label{E:cSFP}
f(x_E)=\int_E f_t(x) f(x_t) d\mu(t).
\end{equation}
 We are primarily interested in the representation of $x$ as 
$x=\int_M f_t(x) x_t d\mu(t)$, and so it may feel tedious to check Equation \eqref{E:cSFP} for all measurable sets $E\in \Sigma$.  However, the following example shows that it is necessary to check \eqref{E:cSFP} for all $E\in \Sigma$ and not just $E=M$.

\begin{example}\label{E:notCSF}
Let $(e_j)_{j\in\N}$ be the unit vector basis for $c_0$ with biorthogonal functionals $(e_j^*)_{j\in\N}$.  Consider the following sequence of pairs in $c_0\times \ell_1$,
$$(x_n,f_n)_{n=1}^\infty=(e_1,e^*_1),(e_1,-e^*_1),(e_1,e^*_1),(e_2,e_2^*),(e_2,-e_1^*),(e_2,e_1^*),(e_3,e_3^*),(e_3,-e_1^*),(e_3,e_1^*),...
$$
If we consider $\N$ with counting measure, then for all $x\in c_0$ and $f\in \ell_1$, $f(x)=\int_\N f_n(x) f(x_n)$.  However, $(x_n,f_n)_{n\in\N}$ is not a continuous Schauder frame.  
Indeed, let $x=e_1$ and suppose $x_{3\N}\in c_0$ is such that for all $f\in \ell_1$, $f(x_{3\N})=\int_{3\N}f_n(e_1)f(x_n)$.  Then for all $n\in\N$, $e_n^*(x_{3\N})=1$.  Thus, $x_{3\N}=(1,1,1,...)\not\in c_0$ which is a contradiction. 
\end{example}

What is particularly interesting about Example \ref{E:notCSF} is that although $(x_n,f_n)_{n\in\N}$ is not a continuous Schauder frame of $c_0$, we do have that the dual $(f_n, x_n)_{n\in\N}$ is a continuous* Schauder frame of $\ell_1$.  Indeed, if $f\in \ell_1$ and $E\subseteq \N$ then
$$f_E=\sum_{n\in M\cap(3\N-2)} f(e_{(n+2)/3})e_{(n+2)/3}^*+\sum_{n\in M\cap 3\N} f(e_{n/3})e_1^*-\sum_{n\in M\cap (3\N-1)} f(e_{(n+1)/3})e_1^*.
$$
As $f\in\ell_1$, we have that all the above series converge in norm to an element of $\ell_1$.  Thus, we have that it is possible to have a continuous* Schauder frame $(f_t, x_t)_{t\in M}$ for a dual space $X^*$ such that $(x_t, f_t)_{t\in M}$ is not a continuous Schauder frame for $X$.

The following lemma shows that any continuous Schauder frame or continuous* Schauder frame satisfies an unconditionality inequality.

\begin{lem}\label{L:suppression}
Let either $(x_t,f_t)_{t\in M}\in X\times X^*$ be a continuous Schauder frame of a Banach space $X$ or $(f_t,x_t)_{t\in M}\in X^*\times X$ be a continuous* Schauder frame for $X^*$.  Then, there exists a constant $B>0$ such that for every measurable set $E$ in $M$ and every $x\in X$ and $f\in X^*$ we have that 
$$|f(x_E)|=|f_E(x)|=\left|\int_E f_t(x) f(x_t) d\mu(t)\right|\leq B \|x\|\|f\|$$
\end{lem}
\begin{proof}
Suppose that either $(x_t,f_t)\in X\times X^*$ is a continuous Schauder frame of $X$ or $(f_t,x_t)\in X^*\times X$ is a continuous* Schauder frame for $X^*$. 
In either case, we have that  $|f(x)|=|\int f_t(x) f(x_t) d\mu(t)|\leq \int |f_t(x) f(x_t)| d\mu(t)<\infty$ for all $x\in X$ and $f\in X^*$. 
  Thus, the linear map, $(x,f)\mapsto f_t(x) f(x_t)$ is pointwise bounded from $X\times X^*$ to $L_1(\mu)$ and is hence uniformly bounded.  Thus, there exists $B>0$ such that 
$\int |f_t(x) f(x_t)| d\mu(t) \leq B \|x\|\|f\|$ for all $(x,f)\in X\times X^*$.  In particular, for all measurable $E$ we have, $|\int_E f_t(x) f(x_t) d\mu(t)|\leq  \int |f_t(x) f(x_t)| d\mu(t) \leq B \|x\|\|f\|$.
\end{proof}

Lemma \ref{L:suppression} gives that for every continuous Schauder frame $(x_t,f_t)\in X\times X^*$ there exists a constant $B_{s}>0$ such that for every measurable set $E$ we have that $\|x_E\|\leq B_s\|x\|$.  The least constant $B_s$ to satisfy Lemma \ref{L:suppression} is called the {\em suppression unconditionality constant} of $(x_t,f_t)_{t\in M}$.  Likewise, the {\em unconditionality constant} of $(x_t,f_t)_{t\in M}$ is the least constant $B_u$ to satisfy $\int |f_t(x) f(x_t)| d\mu(t)\leq B_u|\int f_t(x) f(x_t) d\mu(t)|$ for all $x\in X$ and $f\in X^*$.   Just like unconditional bases, these constants satisfy $B_s\leq B_u\leq 2B_s$.

As duality techniques are ubiquitous in the theory of Banach spaces, it will be important to determine the relationship between a map $t\mapsto (x_t,f_t) \in X\times X^*$ and the dual map $t\mapsto (f_t,x_t) \in X^*\times X$.  The following lemma states that if $(x_t,f_t)_{t\in M}$ is a continuous Schauder frame of $X$ then $(f_t,x_t)_{t\in M}$ is a continuous* Schauder frame of $X^*$.  In Section \ref{S:sbc} we will characterize when $(f_t,x_t)_{t\in M}$ is a continuous Schauder frame of $X^*$ and not just a  continuous* Schauder frame.

\begin{lem}\label{L:DualFrame1} Let $X$ be a separable Banach Space and let $t\mapsto(x_t,f_t)\in X\times X^*$ be a measurable map from a $\sigma$-finite measure space $M$ to $X\times X^*$.  If $(x_t, f_t)_{t \in M}$ is a continuous Schauder frame for $X$ then the \textit{dual frame} $(f_t, x_t)_{t \in M}$ is a continuous* Schauder Frame for $X^{*}$.
\end{lem}
\begin{proof} 
We assume that $(x_t,f_t)_{t\in M}$ is a continuous Schauder frame of $X$.
Fix $f \in X^{*}$ and let $E$ be measurable. We have by Lemma \ref{L:suppression} that the map $x\mapsto \int_E f_t(x) f(x_t) d\mu(t)$ defines a bounded linear functional on $X$.  Thus, there exists $f_E\in X^*$ such that $f_E(x)=\int_E f_t(x) f(x_t) d\mu(t)$ for all $x\in X$.  Furthermore, $f(x)=\int f_t(x) f(x_t) d\mu(t)$ for all $x\in X$ and hence $f_M=f$ and $(f_t, x_t)_{t \in M}$ is a continuous* Schauder Frame for $X^{*}$.
\end{proof}

\section{Shrinking and boundedly complete continuous Schauder frames}\label{S:sbc}

The properties shrinking and boundedly complete give very nice structural results for Schauder bases.  The properties are extended to atomic decompositions and Schauder frames in \cite{CL}, \cite{CLS} and \cite{L}, and they prove that many of the fundamental James Theorems for bases extend to Schauder frames.  The goal for this section is to extend these results to continuous Schauder frames as well.  The natural numbers are used to index Schauder bases and Schauder frames, and so it is very easy to work with properties using limits.  Continuous frames however are indexed by arbitrary measure spaces, and so we will have to work with nets over a directed set instead.

\begin{defn} Given a $\sigma$-algebra $\Sigma$ and measure $\mu$, we introduce \[\mathcal{D} = \{ E \in \Sigma : \mu(E) < \infty\textrm{ and }\sup_{t \in A} \|f_t\|\|x_t\| < \infty \}\] We make $\mathcal{D}$ a directed set by defining $E \preceq F$ whenever $E \subseteq F$, and we refer to the elements of $\mathcal{D}$ as \textit{extra finite}.
\end{defn}

For $E\in\Sigma$, we define bounded operators $P_E,T_E:X\rightarrow X$ by for all $x\in E$
$$P_E(x)=\int_E f_t(x) x_t d\mu(t)\quad\textrm{ and }\quad T_E(x)=\int_{M\setminus E} f_t(x) x_t d\mu(t).
$$
We consider $P_E$ as a generalization of the basis projection for a Schauder basis and $T_E$ as the tail projection.  However, in the setting of Schauder frames and continuous Schauder frames these operators are no longer projections.

\begin{defn}\label{D:S} A continuous Schauder Frame $(x_t, f_t)_{t \in M}$ for $X$ is called \textit{shrinking} if $\lim_{E \in \mathcal{D}} \|T^*_{ E}f\| = 0$ for all $f \in X^{*}$.
\end{defn}

\begin{lem}\label{L:uncS}
Let $(M,\Sigma,\mu)$ be a sigma finite measure space and let $(x_w,f_w)_{w\in M}$ be a continuous Schauder frame of a Banach space $X$.  If $(x_w,f_w)_{w\in M}$ is shrinking then for all $H\in \Sigma$ we have that
$$\lim_{E \in \mathcal{D}} \|T^*_{H \cup E}f\| = 0\quad\textrm{ for all }f \in X^{*}.$$

\end{lem}
\begin{proof}
Let $f\in X^*$ and $\vp>0$.  Assume that $(x_w,f_w)_{w\in M}$ is shrinking.  Thus, there exists $E_0\in \mathcal{D}$ such that $\|T^*_E f\|<\vp$ for all $E\in \mathcal{D}$ with $E_0\subseteq E$. 
Let $x\in B_X$.  We consider the two sets 
$$F^+=\{t\in M \,:\, f_t(x) f(x_t)\geq 0\}\quad\textrm{ and }\quad F^-=\{t\in M\,:\, f_t(x) f(x_t)< 0\}.$$
Choose increasing sequences $(F^+_n)_{n\in\N},(F^-_n)_{n\in\N}\subseteq \mathcal{D}$ such that $\cup F^+_n=F^+$ and $\cup F^-_n=F^-$.  We now have the following estimate for each $x\in B_X$.

\begin{align*}
\int_{E_0^c} |f_t(x) f(x_t)| d\mu&=\int_{(E_0\cup F^-)^c} f_t(x) f(x_t) d\mu-\int_{(E_0\cup F^+)^c} f_t(x) f(x_t) d\mu\\
&=\lim_{n\rightarrow \infty}\int_{(E_0\cup F_n^-)^c} f_t(x) f(x_t) d\mu-\int_{(E_0\cup F_n^+)^c} f_t(x) f(x_t) d\mu\\
&\leq \lim_{n\rightarrow \infty}\|T^*_{E_0\cup F_n^-} f\|+ \|T^*_{E_0\cup F_n^+}f\|\leq   2\vp
\end{align*}

Thus, for all $x\in B_X$ we have that $\int_{E_0^c} |f_t(x) f(x_t)| d\mu<2\vp$.   Let $H\in M$ and $E\in\mathcal{D}$ with $E_0\subseteq E$.  
$$\|T^*_{E\cup H}f\|=\sup_{x\in B_X}\left|\int_{E^c\cap H^c} f_t(x)f(x_t)d\mu\right|\leq \sup_{x\in B_X}\int_{E_0^c} |f_t(x)f(x_t)|d\mu\leq 2\vp
$$
Thus, $\lim_{E \in \mathcal{D}} \|T^*_{H \cup E}f\| = 0$. 
\end{proof}

A Schauder basis for a Banach space is shrinking if and only if its biorthogonal functionals form a Schauder basis for the dual.  The following theorem proves this for continuous Schauder frames, which essentially means that we have the correct definition for what it means for a continuous Schauder frame to be shrinking.

\begin{thm}\label{T:AltShrink} Let  $(x_t, f_t)_{t \in M}$ be a continuous Schauder frame for a Banach space $X$. Then,  $(x_t, f_t)_{t \in M}$ is shrinking if and only if its dual frame $(f_t, x_t)_{t \in M} \subseteq X^{*} \times X^{**}$ is a continuous Schauder Frame for $X^{*}$. 
\end{thm}
\begin{proof}
Let $f\in X^*$.  By Lemma \ref{L:DualFrame1}, $(f_t, x_t)_{t \in M}$ is a continuous* Schauder frame for $X$.  Thus, there exists $f_H\in X^*$ with $f_H(x)=\int_H f_t(x) f(x_t) d\mu(t)$ for all $x\in X$ and all measurable $H\subseteq M$, and $f_M=f$.  We now need to prove that $(x_t, f_t)_{t \in M}$ is shrinking if and only if $x^{**}(f_H)=\int_H x^{**}(f_t) f(x_t) d\mu(t)$ for all $x^{**}\in X^{**}$ and all measurable $H\subseteq M$.

$M$ is $\sigma$-finite so we can find a monotone sequence $(V_i)_{i = 1}^\infty$ of finite measure sets such that $M = \bigcup_{i = 1}^\infty V_i$. We also define $W_i = \{ t \in M : \norm{f_t} \norm{x_t} \leq i \}$ and note that $(W_i)_{i = 1}^\infty$ is monotone with $M = \bigcup_{i = 1}^\infty W_i$. Now let $U_i = V_i \cap W_i$. We see that $(U_i)_{i = 1}^\infty \subseteq \mathcal{D}$ is monotone and $M = \bigcup_{i = 1}^\infty U_i$.

We now fix $x^{**} \in X^{**}$  and choose $(y_n)_{n = 1}^\infty \subseteq X$ converging weak$^{*}$ to $x^{**}$ such that $\|y_n\|= \|x^{**}\|$ for all $n\in\N$.  In particular, we have for all $n,k\in\N$ that $|x_t(f) y_n(f_t)|\leq k^2\|f\|\|x^{**}\|$ for all $t\in M$.  Thus, the function $t\mapsto x_t(f) y_n(f_t)$ is bounded by the constant $k^2\|f\|\|x^{**}\|$ on the finite measure space $U_k$.  Let $H\in \Sigma$.  For each $k\in \N$, we have that

\begin{align*} x^{**}(f_H) =& \lim_{n \rightarrow \infty} f_H(y_n) \\
 =& \lim_{n \rightarrow \infty} \int_H f(x_t) f_t(y_n) d\mu  \\
 =& \lim_{n \rightarrow \infty}\int_{H \cap U_k} f(x_t) f_t(y_n) d\mu + \lim_{n \rightarrow \infty} \int_{H \cap U_k^C} f(x_t) f_t(y_n) d\mu \\
 =&  \int_{H \cap U_k} f(x_t) x^{**}(f_t) d\mu + \lim_{n \rightarrow \infty} \int_{H \cap U_k^C} f(x_t) f_t(y_n) d\mu \quad \textrm{by dominated convergence.} \end{align*} 
Now as $x_t(f) x^{**}(f_t)\in L_1(H)$ we have that,

 \begin{align*}
  x^{**}(f_H)  =& \lim_{k\rightarrow\infty}  \int_{H \cap U_k} f(x_t) x^{**}(f_t) d\mu +\lim_{k\rightarrow\infty} \lim_{n \rightarrow \infty} \int_{H \cap U_k^C} f(x_t) f_t(y_n) d\mu \\
=& \int_{H} f(x_t) x^{**}(f_t) d\mu +\lim_{k\rightarrow\infty} \lim_{n \rightarrow \infty} \int_{H \cap U_k^C} f(x_t) f_t(y_n) d\mu \quad \textrm{by dominated convergence.} 
 \end{align*}

Thus, we have that $(f_t, x_t)_{t \in M}$ is a continuous Schauder Frame for $X^{*}$ if and only if $\lim_{k \rightarrow \infty}\lim_{n \rightarrow \infty} \int_{H  \cap U_k^C} f(x_t) f_t(y_n) d\mu = 0$ for every $f \in X^{*}$, $H \in \Sigma$, and $w$-Cauchy $(y_n)_{n=1}^\infty$ in $X$. 

We first assume that $(x_w,f_w)_{w\in M}$ is shrinking.  Let $\varepsilon > 0$, $f \in X^{*}$ with $\|f\|=1$, and $H \in \Sigma$. By Lemma \ref{L:uncS}, there exists $E \in \mathcal{D}$ such that for all $F \succeq E$ we have that $\norm{T^*_{H^C \cup F}f} \leq \varepsilon$. This $E$ is extra finite and so let $A = \sup_{t \in E} \norm{f_t} \norm{x_t} < \infty$. As $E$ has finite measure and $M=\cup_{j=1}^\infty U_j$, there exists $k \in \mathbb{N}$ so that  $\mu(H \cap U_k^C \cap E) < \frac{\varepsilon}{A}$.  Then:

 \begin{align*}\lim_{n \rightarrow \infty} &\left |\int_{H \cap U_k^C} f(x_t) f_t(y_n) d\mu \right| \leq \|x^{**}\|\sup_{\norm{x} = 1} \left|\int_{H \cap U_k^C} f(x_t) f_t(x) d\mu\right| \\
  & \leq \|x^{**}\|\sup_{\norm{x} = 1} \left|\int_{H \cap U_k^C \cap E^C} x_t(f) f_t(x) d\mu\right| + \|x^{**}\|\sup_{\norm{x} = 1} \left|\int_{H \cap U_k^C \cap E} x_t(f) f_t(x) d\mu\right| \\
   &\leq\|x^{**}\| \norm{T^*_{H^C \cup U_k \cup E}f} + \|x^{**}\| \mu(H \cap U_k^C \cap E) A \\
    &\leq \|x^{**}\| \varepsilon+\|x^{**}\|\varepsilon 
    \end{align*}
    
     Thus, $\lim_{k \rightarrow \infty}\lim_{n \rightarrow \infty} \int_{H  \cap U_k^C} x_t(f) y_n(f_t) d\mu = 0$ and thus $(f_t, x_t)_{t \in M}$ is a continuous Schauder Frame for $X^{*}$.

We now assume that  $(f_t, x_t)_{t \in M}$ is not a continuous Schauder frame of $X^*$.  Thus, there exists $f \in X^{*}$, $H \in \Sigma$, $(m_k)_{k=1}^\infty\in[\N]^\omega$, $\vp>0$, and a $w^*$-convergent sequence $y_n\rightarrow x^{**}$ with $\|y_n\|=x^{**}$ for all $n\in\N$ such that $\lim_{n \rightarrow \infty} \int_{H  \cap U_{m_k}^C} x_t(f) y_n(f_t) d\mu >\vp$ for all $k\in\N$.   

Let $E\in \mathcal{D}$.  let $A = \sup_{t \in E} \norm{f_t} \norm{x_t} < \infty$. As $E$ has finite measure and $M=\cup_{j=1}^\infty U_j$, there exists $k \in \mathbb{N}$ so that  $\mu(H \cap U_{m_k}^C \cap E) < \frac{\varepsilon}{2A\|f\|\|x^{**}\|}$.  Then:

\begin{align*}
\|T^*_{H\cup U_k\cup E}f\|&=\sup_{\|x\|=1} \int_{H^c\cap U_{m_k}^c \cap E^c} f_t(x) f(x_t) d\mu\\
&\geq \lim_{n\rightarrow\infty} \frac{1}{\|x^{**}\|} \int_{H^c\cap U_{m_k}^c \cap E^c} f_t(y_n) f(x_t) d\mu\\
&= \lim_{n\rightarrow\infty} \frac{1}{\|x^{**}\|}\left( \int_{H^c\cap U_{m_k}^c} f_t(y_n) f(x_t) d\mu- \int_{H^c\cap U_{m_k}^c \cap E} f_t(y_n) f(x_t) d\mu\right)\\
&\geq \lim_{n\rightarrow\infty} \frac{1}{\|x^{**}\|}\left( \int_{H^c\cap U_{m_k}^c} f_t(y_n) f(x_t) d\mu- A\|x^{**}\|\|f\|\mu(H^c\cap U_{m_k}^c \cap E)\right)\\
&\geq \lim_{n\rightarrow\infty} \frac{1}{\|x^{**}\|}\left( \int_{H^c\cap U_{m_k}^c} f_t(y_n) f(x_t) d\mu- \vp/2\right)\\
&\geq \frac{1}{\|x^{**}\|}\frac{\vp}{2}
\end{align*}
Thus, $(x_w,f_w)_{w\in M}$ is not shrinking by Lemma \ref{L:uncS} as $E\in\mathcal{D}$ is arbitrary and $E\preceq U_k\cup E$.  

\end{proof}

\begin{thm}\label{T:l1}Let  $(x_t, f_t)_{t \in M}$ be a continuous Schauder frame for a Banach space $X$.  Then  $(x_t, f_t)_{t \in M}$ is shrinking if and only if  $\ell_1$ does not embed into $X$.
\end{thm}
\begin{proof}

We first assume that $(x_t, f_t)_{t \in M}$ is shrinking.  Thus, $(f_t, x_t)_{t \in M}$ is a continuous Schauder frame for $X^*$ by Theorem \ref{T:AltShrink}.  In particular, $X^*$ must be separable.  However, the dual of $\ell_1$ is not separable, so $\ell_1$ cannot embed into $X$.

We now assume that $(x_t, f_t)_{t \in M}$ is not shrinking.
 There exists $f \in X^{*}$ with $\|f\|=1$ and $A>0$ such that, for every $E \in \mathcal{D}$ there is some $F \succeq E$ such that $\norm{T^*_{F} f} > A$.  Let $B$ be the suppression unconditionality constant of  $(x_t, f_t)_{t \in M}$. 
 
 We will create by induction a sequence of extra finite sets $U_1 \subseteq V_1\subseteq U_2\subseteq V_2...$ and vectors $(y_n)_{n=1}^\infty$ such that for all $n\in\N$,
 
 \begin{enumerate}
 \item $\|y_n\|\leq 2$,
 \item $P^*_{V_n\setminus U_n} f(y_n)> A/2$,
 \item $|P^*_{V_j\setminus U_j} f(y_n)|<\frac{A}{9nB}$ for all $j<n$,
 \item $|P^*_{V_n\setminus U_n} f(y_j)|<\frac{A}{92^{n}B}$ for all $j<n$.
 \end{enumerate}
 
 We first choose $U_1=\emptyset$. There exists $E\in\mathcal{D}$ such that $\|T_E^* f\|>A$.  Let $y_1\in X$ be a unit vector such that $T_E^*f(y_1)>A$. We choose $V_1\subseteq E^c$ such that $P_{V_1}f(y_1)>A$.
 
 We now assume that $n\in\N$ and that $y_1,...y_n$ and $U_1 \subseteq V_1...\subseteq U_n\subseteq V_n$ have been chosen.  As $(x_t, f_t)_{t \in M}$ is a continuous Schauder frame for $X$ there exists $U\in\mathcal{D}$ such that for all $E\subset U^c$, we have that $|P_E^*f(y_j)|<\frac{A}{92^{n+1}B}$ for all $j\leq n$. 
 We have that for every $E \in \mathcal{D}$ there is some $F \succeq E$ and $z\in X$ with $\|z\|=1$ such that $|{T^*_{F} f}(z)| > A$.  By compactness, we may stabilize the values $|P^*_{V_j\setminus U_j} f(z)|$ for $1\leq j\leq n$.  
 That is, for all $E_1 \in \mathcal{D}$ there is some $F_1 \succeq E_1$ and $z_1\in X$ with $\|z_1\|=1$ such that for all $E_2 \in \mathcal{D}$ there is some $F_2 \succeq E_2$ and $z_2\in X$ with $\|z_2\|=1$ so the following is satisfied for all $1\leq j\leq n$, 
 \begin{equation}\label{E:ind}
 |{T^*_{F_1} f}(z_1)| > A,\quad |{T^*_{F_2} f}(z_2)| > A,\quad\textrm{and}\quad |P^*_{V_j\setminus U_j} f(z_1-z_2)|<\frac{A}{9(n+1)B}.
 \end{equation}
   We 
 choose $F_1\succeq U \cup V_n$ and obtain $z_1$ satisfying \eqref{E:ind}.  
 We choose $E_2\in\mathcal{D}$ so that $|{T^*_{F} f}(z_1)|<A/2$ for all $F\succeq E_2$. We now choose $F_2\succeq E_2$ and $z_2$ satisfying \eqref{E:ind}.  Let $U_{n+1}=E_2$ and $V_{n+1}\succeq U_{n+1}$ such that $P^*_{V_{n+1}\setminus U_{n+1}} f(z_2)> A$.
  We now let $y_{n+1}=z_2-z_1$.  These choices now satisfy our induction proof.

We have that $(y_n)_{n=1}^\infty$ is semi-normalized by (1) and (2).  We now prove that $(y_n)_{n=1}^\infty$  is equivalent to the unit vector basis of $\ell_1$.  Let $(a_n)_{n=1}^\infty\in\ell_1$ with $\sum |a_n|=1$.  Without loss of generality, we assume there is a subsequence $(a_{k_n})_{n=1}^\infty$ such that $a_{k_n}> 0$ for all $n\in\N$ and $\sum_{n=1}^\infty a_{k_n}\geq \frac{1}{2}$.

\begin{align*}
\|\sum_{n=1}^\infty a_n y_n\|&\geq B^{-1} (P^*_{\cup_{j\in\N} V_{k_j}\setminus U_{k_j}} f)(\sum_{n=1}^\infty a_n y_n)\\
&= B^{-1} \sum_{n=1}^\infty a_{k_n} P^*_{V_{k_n}\setminus U_{k_n}} f (y_n)+\sum_{j<n}a_n P^*_{V_{k_j}\setminus U_{k_j}} f( y_n)+
\sum_{j>n}a_n P^*_{V_{k_j}\setminus U_{k_j}} f( y_n)\\
&\geq B^{-1} \sum_{n=1}^\infty a_{k_n} A/2-\sum_{j<n} |a_n| n^{-1}B^{-1}9^{-1}A-
\sum_{j>n}|a_n| 2^{-j}B^{-1}9^{-j}A\\
&\geq B^{-1} \sum_{n=1}^\infty a_{k_n} A/2-\sum_{n\in\N} |a_n| B^{-1}9^{-1}A-
\sum_{n\in\N}|a_n| B^{-1}9^{-1}A\\
&\geq B^{-1} A/4-B^{-1}9^{-1}A-B^{-1}9^{-1}2^{-j}A=B^{-1}A36^{-1}
\end{align*}
Thus, $(y_n)_{n=1}^\infty$ dominates the unit vector basis of $\ell_1$ and is hence equivalent to it.
\end{proof}

We now consider the generalization of boundedly complete to the continuous setting.

\begin{defn} A continuous Schauder Frame $(x_t, f_t)_{t \in M}$ for $X$ is said to be \textit{boundedly complete} if $\lim_{E \in \mathcal{D}} \int_E x^{**}(f_t) x_t d\mu \in X$ for all $x^{**} \in X^{**}$. 
\end{defn}

Frames for Hilbert spaces are nicely characterized as projections of Riesz bases for larger Hilbert spaces.  Likewise, Schauder frames for Banach spaces are characterized as projections of Schauder bases for larger Banach spaces.  Furthermore, a Schauder frame is shrinking or boundedly complete if and only if it is the projection of a shrinking or boundedly complete Schauder basis \cite{BFL}.  This allows for constructing and studying frames by working directly with bases and then projecting onto a subspace.  Essentially, a redundant frame may be dilated to a non-redundant basis. However, this concept of dilation is only possible for continuous frames over purely atomic measures.   The following proposition shows that the reverse direction is still valid for continuous frames in that projecting continuous Schauder frames onto closed subspaces gives a continuous Schauder frame.
\begin{prop}\label{P:ProjF}
Let $(x_t, f_t)_{t}$ be a continuous Schauder frame for a Banach space $X$.  Let $Y\subseteq X$ be a complemented subspace and let $P:X\rightarrow Y$ be a bounded projection.
\begin{enumerate}
\item $(P x_t, P^* f_t)_{t}$ is a Schauder frame for $Y$.
\item If $(x_t, f_t)_{t}$ is shrinking then $(P x_t, P^* f_t)_{t}$ is shrinking.
\item If $(x_t, f_t)_{t}$ is boundedly complete then $(P x_t, P^* f_t)_{t}$ is boundedly complete.
\end{enumerate}
\end{prop}

\begin{proof}
Let $y\in Y$ and $g\in Y^*$.  Let $E\subseteq M$ be measurable.   There exists $y_E\in X$ such that $y_E=\int f_t(y) x_t dt$.  Thus,
$$g(P y_E)=P^* g(y_E)=\int_E f_t(y) P^*g(x_t)dt=\int_E P^*f_t(y) g(Px_t)dt.
$$
Thus, $P y_E= \int_E f_t(y) x_t dt$.  Furthermore, $P y_M=Py=y$.  This proves that $(P x_t, P^* f_t)_{t}$ is a Schauder frame for $Y$.

We now assume that $(x_t, f_t)_{t}$ is shrinking.   Let $E\subseteq M$ be measurable.  If $P_E$ is the projection operator for the Schauder frame $(x_t, f_t)_{t}$ of $X$  then we have that $P P_E$ is the projection operator for the frame 
$(Px_t, P^*f_t)_{t}$ of $Y$.  Thus, if $T_E$ is the tail operator for the Schauder frame $(x_t, f_t)_{t}$ of $X$  then $P T_E$ is the tail operator for the frame 
$(Px_t, P^*f_t)_{t}$ of $Y$.  This gives that for all $g\in Y^*$ that
 $$\lim_{E \in \mathcal{D}} \|(PT_E)^*g\|=\lim_{E \in \mathcal{D}} \|T_E^* (P^*g)\|  = 0 .$$
Thus, $(Px_t, P^*f_t)_{t}$ is shrinking.

We now assume that  $(x_t, f_t)_{t}$ is boundedly complete.  Let $y^{**}\in Y^{**}$.  Let $I_X:Y\rightarrow X$ be the inclusion operator of $Y$ into $X$.  Thus, $I_X^{**}y^{**}\in X^{**}$ and there exists $x\in X$ such that $x=\int I_X^{**}y^{**}(f_t) x_t dt$. Let $E\subseteq M$ be measurable and $f\in X^{**}$.  We have that,
\begin{align*}
g(P(x_E))&=\int_E I^{**}y^{**}(f_t) P^*g^*(x_t) dt \\
&=\int_E y^{**}(f_t|_Y) g^*(Px_t) dt \\
&=\int_E y^{**}(P^*f_t) g^*(Px_t) dt \\
\end{align*}
Thus, we have that $Px=\int y^{**}(P^*f_t) Px_t dt$.  Hence, $(Px_t, P^*f_t)_{t}$ is boundedly complete.

\end{proof}

\begin{lem}\label{L:inX}
Let $(M,\Sigma,\mu)$ be a $\sigma$-finite measure space and $(x_t, f_t)_{t \in M}$ be a continuous Schauder frame for $X$.  For each $E \in \mathcal{D}$, the operator $P^{**}_E$ on $X^{**}$ defined by $P_E(x^{**})=\int_E x^{**}(f_t)x_t d\mu(t)$ is compact, has its range in $X$, and satisfies $\|P^{**}_E\|\leq B$ where $B$ is the suppression unconditionality constant of $(x_t, f_t)_{t \in M}$.
\end{lem}

\begin{proof}
We denote $L(X^{**},X)$ to be the set of bounded linear operators from $X^{**}$ to $X$.  We have for each $t\in M$ that $f_t \otimes x_t\in  L(X^{**},X)$, where $f_t \otimes x_t(x^{**})=f_t(x^{**})x_t$ for all $t\in M$. 
Let $\psi:M\rightarrow  L(X^{**},X)$ be the measurable map $t\mapsto f_t\otimes x_t$.
 As $(M,\Sigma,\mu)$ is $\sigma$-finite, if $(H_\alpha)$ is a measurable collection of pairwise disjoint subsets of $L(X^{**},X)$ then there exists a countable subset such that $\mu(\psi^{-1}(H_\alpha))=0$ for all $\alpha$ not in the subset.   Thus, for each $\vp>0$, there exists a countable collection of sets $(H_j)_{j\in J}$ in $L(X^{**},X)$ with diameter at most $\vp$ such that $\cup_{j\in J} \psi^{-1}(H_j)=E$ almost everywhere and  $\psi^{-1}(H_j)\neq\emptyset$ for all $j\in J$.   For all $j\in J$ choose $t_j\in \psi^{-1}(H_j)$.   Let $C=\sup_{t\in E}\|x_t\|\|f_t\|$.
We have that 
$$ \sum_{j\in J}\|f_{t_j}\otimes x_{t_j}\| \mu(\psi^{-1}(H_j))\leq \sum_{j\in J} C  \mu(\psi^{-1}(H_j))= C\mu(E).
$$
 Thus, $\sum_{j\in J} f_{t_j}\otimes x_{t_j} \mu(\psi^{-1}(H_j))$ converges unconditionaly in norm.  As, $ f_{t_j}\otimes x_{t_j}\in L(X^{**},X)$ for all $j\in J$, we have that $T_\vp:=\sum_{j\in J} f_{t_j}\otimes x_{t_j} \mu(\psi^{-1}(H_j))$ is an element of $L(X^{**},X)$.  For each $x^{**}\in X^{**}$ and $f\in X^*$, we have that 
 \begin{align*}
 |\int_E x^{**}(f_t) &f(x_t)d\mu - \sum_{j\in J} x^{**}(f_{t_j})f(x_{t_j})\mu(H_j)|\\
 &=| \sum_{j\in J} \left(\int_{\psi^{-1}(H_j)} x^{**}(f_t) f(x_t)d\mu -x^{**}(f_{t_j})f(x_{t_j})\mu(H_j)\right)|\\
 &\leq \sum_{j\in J} \int_{\psi^{-1}(H_j)} \|x^{**}\|_{X^{**}}\|f\|_{X^*}\|f_t\otimes x_t- f_{t_j}\otimes x_{t_j}\| d\mu\\
&\leq \sum_{j\in\N} \int_{\psi^{-1}(H_j)} \|x^{**}\|_{X^{**}}\|f\|_{X^*}\vp  d\mu\\
&=   \int_{E} \|x^{**}\|_{X^{**}}\|f\|_{X^*}\vp  d\mu\\
&= \|x^{**}\|_{X^{**}}\|f\|_{X^*}\vp  \mu(E)
 \end{align*}
 
 Thus,  $(T_{1/n})_{n=1}^\infty$ is Cauchy and  it converges in norm to $P^{**}_E=\int_E f_t \otimes x_t d\mu$.  
As the range of $T_{1/n}$ is in $X$ for all $n\in\N$, we have that  the range of $P^{**}_E$ is in $X$. Note that $P_E^{**}=(P_E)^{**}$, the double adjoint of the truncation operator for the continuous Schauder frame $(x_t, f_t)_{t \in M}$.   Thus, $\|P_E^{**}\|=\|P_E\|\leq B$.
 \end{proof}

\begin{thm}\label{T:bc}
Let $(x_t, f_t)_{t\in M}$ be a continuous Schauder frame for a Banach space $X$.  Then $(x_t, f_t)_{t\in M}$ is boundedly complete if and only if $c_0$ does not embed into $X$.
\end{thm}
\begin{proof}
We first assume that $(x_t, f_t)_{t\in M}$ is not boundedly complete.  Thus, there exists $x^{**}\in X^{**}$ such that 
$\lim_{E \in \mathcal{D}} \int_E x^{**}(f_t) x_t d\mu$ does not converge to an element of $X$.  We have by Lemma \ref{L:inX} that, $\int_E x^{**}(f_t) x_t d\mu$ is an element of $X$ for all $E\in\mathcal{D}$.  Hence, the net $\lim_{E \in \mathcal{D}} \int_E x^{**}(f_t) x_t d\mu$ is not Cauchy.  This gives that there exists $\delta > 0$ and extra finite sets $V_1 \subseteq W_1 \subseteq V_2 \subseteq W_2 \subseteq ...$ such that for $u_n =\int_{W_n - V_n} x^{**}_0(f_t) x_t d\mu$, we have $\norm{u_n} > \delta$ for every $k \in \mathbb{N}$.

Let $f\in X^*$.  As, $|\int_{M} x^{**}_0(f_t) f(x_t) d\mu|<\infty$ and $(W_n-V_n)_{k=1}^\infty$ is pairwise disjoint,  we have that 
$$\lim_{n\rightarrow\infty}f(u_n)=\int_{W_n - V_n} x^{**}_0(f_t) f(x_t) d\mu=0.$$
Thus, $(u_n)_{n=1}^\infty$ is a semi-normalized weakly null sequence.  After passing to a subsequence, we assume without loss of generality that $(u_n)_{n=1}^\infty$ is a basic sequence.  We will now prove that $(u_n)_{n=1}^\infty$ is isomorphic to the unit vector basis of $c_0$.  Let $C>0$ be the unconditionality constant of $(x_t, f_t)_{t\in M}$.  Let $f\in X^*$ and $(a_n)_{n\in\N}\in c_{00}$.  We have that 
$$|f(\sum a_n u_n)|=|\sum_{n\in\N}\int_{W_n-V_n} a_n x^{**}_0 (f_t) f(x_t) d\mu|\leq\sup |a_n| \int |x^{**}_0 (f_t) f(x_t)|d\mu \leq C\|f\|\|a_n\|_\infty 
$$
Thus, $(u_n)_{n\in\N}$ is equivalent to the unit vector basis of $c_0$.

We now assume  that $c_0$ is isomorphic to a subspace of $X$, and for the sake of contradiction we assume that $X$ has a boundedly complete continuous Schauder frame.  As $X$ is separable, every subspace isomorphic to $c_0$ is complemented in $X$ (see \cite{Go} for a nice proof of this fact).  Thus, there exists a boundedly complete continuous Schauder frame $(x_t ,f_t)_{t\in M}$ of $c_0$ by Theorem \ref{P:ProjF}.  As $(x_t ,f_t)_{t\in M}$ is boundedly complete, we may define $P:\ell_\infty\rightarrow c_0$ by 
$$P(x^{**})=\lim_{E \in \mathcal{D}} \int_E x^{**}(f_t) x_t d\mu\quad\textrm{ for all }x^{**}\in \ell_\infty.
$$
This gives a bounded linear projection from $\ell_\infty$ to $c_0\subseteq\ell_\infty$.  This is a contradiction as $c_0$ is not complemented inside $\ell_\infty$.

\end{proof}

\begin{thm} If a Banach Space $X$ admits a continuous Schauder Frame $(x_t, f_t)_{t \in M}$ then the following are equivalent: 
\begin{enumerate}
\item $(x_t, f_t)_{t \in M}$ is  shrinking and boundedly complete,\\
\item $X$ does not contain a copy of $c_0$ or $\ell_1$,\\
\item $X$ is reflexive.
\end{enumerate}
\end{thm}

\begin{proof}
We have that (1) and (2) are equivalent by Theorem \ref{T:l1} and Theorem \ref{T:bc}.  Furthermore, $(3) \implies (2)$ is clear since $c_0$ and $\ell_1$ are not reflexive. We now prove that $(1) \implies (3)$.

 Since $(x_t, f_t)_{t \in M}$ is boundedly complete, for each $x^{**} \in X^{**}$ we have an $x \in X$ such that $x = \lim_{A \in \mathcal{D}} \int_A x^{**}(f_t) x_t d\mu$.  As $(x_t,f_t)$ is shrinking, Theorem \ref{T:AltShrink} gives that
every $f \in X^{*}$ satisfies $f = \int x_t(f) f_t d\mu$. Then take arbitrary $f \in X^{*}$ and observe
 \begin{align*}
 x(f) =& f(x) = f(\lim_{A \in \mathcal{D}} \int_A x^{**}(f_t) x_t d\mu) \quad \textrm{by boundedly complete}, \\ 
 =& \lim_{A \in \mathcal{D}} f( \int_A x^{**}(f_t) x_t d\mu) \quad \textrm{by continuity}, \\
  =&  \lim_{A \in \mathcal{D}} \int_A x^{**}(f_t) x_t(f) d\mu \quad \textrm{by definition of the Pettis Integral}, \\
   =&  \int x^{**}(f_t) x_t(f) d\mu \quad \textrm{by dominated convergence.}
   \end{align*}
    On the other hand we have 
    \begin{align*} x^{**}(f) =& x^{**}(\int x_t(f) f_t d\mu) \quad \textrm{by Theorem \ref{T:AltShrink},} \\ 
    =& \int x^{**}(f_t) x_t(f) d\mu \quad \textrm{by definition of the Pettis Integral.}
     \end{align*}
 Compiling the above, we have $x(f) = x^{**}(f)$ for all $f \in X^{*}$. So $x^{**} = x \in X$. But this was for arbitrary $x^{**} \in X^{**}$. So $X^{**} = X$ as desired.
\end{proof}

\section{Continuous wavelet frames for $L_p$ with $1<p<\infty$}\label{S:ex}

Frame theory for Hilbert spaces developed concurrently with that of wavelets, and wavelets still provide some of the most useful examples of both continuous and discrete frames.  Wavelets are important in Banach spaces as well, and the Haar basis is possibly the most commonly used Schauder basis for $L_p$ with $1\leq p<\infty$.    One nice aspect of using a continuous wavelet frame as opposed to a discrete one is that when a function is translated or dilated, the continuous wavelet frame coefficients for the new function are the same (except occurring at different indexes).  On the other hand, if we translate or dilate a function by a non-integer amount, then the discrete wavelet frame  coefficients for the new function are completely different.  It is well known that any discrete wavelet for a Hilbert space gives a continuous wavelet \cite{WW}.  The goal of this section is to prove the corresponding result for $L_p$ for $1<p<\infty$.  We must give a completely different proof than that was used in proving the $L_2$ result as the Fourier transform is not an isometry on $L_p$ for $p\neq2$.

Let $1< p<\infty$ and $\psi\in L_p(\R)$.  For $a,b\in\R$, the dilation operator $D_a:L_p\rightarrow L_p$ and translation operator $T_b:L_p\rightarrow L_p$ are defined by $D_a(f)(t)=2^{a/p}f(2^at)$ and $T_b(f)(t)=f(t-b)$ for all $f\in L_p$ and $t\in \R$.
We call $\psi\in L_p(\R)$ with $\|\psi\|=1$ a wavelet for $L_p$ if $(D_nT_k\psi)_{k,n\in Z}$ is an unconditional Schauder basis of $L_p$.  If $\psi^*\in L_{p'}$ is the biorthogonal functional to $\psi=D_0T_0\psi$ in $(D_nT_k\psi)_{k,n\in Z}$.  Then for all $k,n\in\Z$, the biorthogonal functional to $D_nT_k\psi$ is given by $D^*_{-n}T^*_{-k}\psi^*$.  One can also think of $D^*_{-n}=D_{n}$ and $T^*_{-k}=T_k$ where $D_n$ and $T_k$ are now considered as the dilation and translation operators on $L_{p'}$ for $1/p+1/p'=1$.
 We say that $\psi$ is a continuous wavelet for $L_p$ if $(D_{a}T_b\psi,D^*_{-a}T^*_{-b}\psi^*)_{a,b\in\R}$ is a continuous frame of $L_p$.

The operators $T_a$ and $D_b$ have the following relationships.   For all $a,b\in\R$, $T_a^{-1}=T_{-a}$, $D_a^{-1}=D_{-a}$, $T_aT_b=T_{a+b}$, $D_aD_b=D_{a+b}$, $D_aT_b=T_{2^{-a}b}D_a$ and $D^*_aT^*_b=T^*_{2^{a}b}D^*_a$. For the sake of convenience, we write $\psi_{a,b}=D_aT_b\psi$ and $\psi^*_{a,b}=D^*_{-a}T^*_{-b}\psi$ where $a,b\in\R$.

\begin{lem}\label{L:limit}
Let $(M,\Sigma,\mu)$ be a sigma finite measure space and $X$ be a Banach space. Let $t\mapsto (x_t,f_t)$ be a measurable function from $M$ to $X\times X^*$.   If $((x^n_t,f^n_t)_{t\in M})_{n=1}^\infty$ is a sequence of continuous frames of $X$ with unconditionality constant $C$ such that there exists $D>0$ such that $\|x^n_t\|,\|f^n_t\|\leq D$ for all $t\in M$ and $n\in\N$.  Suppose that
for all $x\in X$ and $f\in X^*$ there exists an increasing sequence of finite measure sets $(E_N)_{N=1}^\infty \subseteq \Sigma$ with $\cup_{N=1}^\infty E_n=M$ so that  for all $N\in\N$.
\begin{enumerate}
\item[(a)] $\lim_{n\rightarrow\infty} \|x-\int_{E_N} f_t^n(x)x_t^n dt\|\leq 1/N$,
\item[(b)] $\lim_{n\rightarrow\infty} \int_{E_N} |(f_t-f_t^n)(x)|dt=0$,
\item[(c)] $\lim_{n\rightarrow\infty} \int_{E_N} | f(x^n_t-x_t)|dt=0$.
\end{enumerate}
Then $(x_t,f_t)_{t\in M}$ is a continuous frame of $X$ with unconditionality constant $C$.
\end{lem}

\begin{proof}
Let $x\in X$ and $f\in X^*$ .  For $E\in\Sigma$ we have that 
\begin{align*}
&|\int_E f_t(x)f(x_t)dt|\leq \int_E |f_t(x)f(x_t)dt|\\
&= \lim_{N\rightarrow\infty}\int_{E\cap E_N} |f_t(x)f(x_t)|dt\quad\textrm{ by continuity from below}\\
&\leq  \lim_{N\rightarrow\infty}\lim_{n\rightarrow\infty}\int_{E_N} | f_t(x)f(x^n_t-x_t)|dt+ \int_{E_N} | (f_t-f_t^n)(x)f(x^n_t)|dt+\int_{E_N} | f_t^n(x)f(x^n_t)|dt\\
&\leq \lim_{N\rightarrow\infty} \lim_{n\rightarrow\infty} \|f_t\|\|x\|\int_{E_N} |f(x^n_t-x_t)|dt + \|f\|\|x^n_t\|\int_{E_N} |(f_t-f_t^n)(x)|dt+\int_{ E_N} | f_t^n(x)f(x^n_t)dt|\\
&\leq 0+0 + C\|f\|\|x\|
\end{align*}
This proves that $x\mapsto \int_E f_t(x) x_t dt$ defines a bounded linear functional on $X^*$ with norm at most $C\|x\|$.  As $X$ is reflexive, there exists $x_E\in X$ with $\|x_E\|\leq C\|x\|$ so that $f(x_E)= \int_E f_t(x) f(x_t) dt$ for all $f\in X^*$.

We now check that $x=x_M$.

\begin{align*}
&|f(x)\!-\!\int\! f_t(x)f(x_t)dt|=\lim_{N\rightarrow\infty} |f(x)\!-\!\int_{E_N}\! f_t(x)f(x_t)dt|\\
&\leq\lim_{N\rightarrow\infty} \lim_{n\rightarrow\infty}\! \int_{E_N}\! |f_t(x)f(x^n_t-x_t)|dt\!+\! \int_{E_N}\! | (f_t-f_t^n)(x)f(x^n_t)|dt\!+\!|f(x)\!-\!\int_{E_N}\! f_t^n(x)f(x^n_t)dt|\\
&\leq\lim_{N\rightarrow\infty} \lim_{n\rightarrow\infty} \|f_t\|\|x\|\int_{E_N} |f(x^n_t-x_t)|dt + \|f\|\|x^n_t\|\int_{E_N} |(f_t-f_t^n)(x)|dt+|f(x)-\int_{E_N}  f_t^n(x)f(x^n_t)dt|\\
&=\lim_{N\rightarrow\infty} 0+0 + 1/N=0
\end{align*}

\end{proof}

\begin{thm}
Let $1< p<\infty$.  Suppose that $\psi$ is a wavelet for $L_p$ with unconditonality constant $C$.  Then $\psi$ is a continuous wavelet for $L_p$ with  unconditionality constant $C$.
\end{thm}
\begin{proof}
 For each $N\in\N$ and $a,b\in\R$  let $l_a,m_b\in\Z$ and $r_a,s_b\in\{0,1,...,N-1\}$ such that $l_a+r_a N^{-1}\leq a<l_a+(r_a+1)N^{-1}$ and $m_b+s_b N^{-1}\leq b<m_b+(s_b+1)N^{-1}$.  
We now 
set 
$$\psi^N_{a,b}=\psi_{l_a+\frac{r_a }{N},m_b+\frac{s_b 2^{l_a}}{N}}.$$

For all $f\in L_{p'}$ and $N\in\N$, we have that, 
$ |f(\psi^{N}_{a,b})-f(\psi_{a,b})| \leq  2\|\psi\|\|x\|.
$
The map $(a,b)\mapsto \psi_{a,b}$ is continuous in norm.  Thus, for all $a,b\in\R$, $\psi^N_{a,b}\rightarrow \psi_{a,b}$ in norm (though we actually only need weak convergence).  We have by the dominated convergence theorem that for all $M\in\N$,
\begin{equation}\label{E:conv2}
\lim_{N\rightarrow\infty}\int_{-M}^M \int_{-M}^M |f(\psi^{N}_{a,b})-f(\psi_{a,b})| dadb=0.
\end{equation}
Likewise, for all $x\in L_p$ and $M\in\N$,
\begin{equation}\label{E:conv}
\lim_{N\rightarrow\infty}\int_{-M}^M \int_{-M}^M |\psi^{N*}_{a,b}(x)-\psi^*_{a,b}(x)| dadb=0.
\end{equation}

Now, let $x\in L_p$ and $\vp>0$. The set $\{ D_aT_bx:|a|,|b|\leq 1\}$ is compact in $L_p$ and thus there exists
 $M\in\N$ such that $\|D_aT_bx-\sum_{n,k=-M}^{M-1} \psi^*_{n,k}(D_aT_bx)\psi_{n,k}\|<\vp$ for all $|a|,|b|\leq1$.  

\begin{align*}
&\|x-\int_{-M}^M\int_{-M}^M\psi^{N*}_{a,b}(x)\psi^N_{a,b}dadb\|=\|x-\sum_{r,s=0}^{N-1}\sum_{l,m=-M}^M N^{-2} \psi^{*}_{l+\frac{r}{N},m+\frac{s 2^{l}}{N}}(x)\psi_{l+\frac rN,m+\frac{s 2^{l}}{N}}\|\\
&=\|x-N^{-2}\sum_{r,s=0}^{N-1}\sum_{l,m=-M}^{M-1}  D^*_{-l-\frac{r}{N}}T^*_{-m-\frac{s2^l}{N}}\psi^*(x) D_{l+\frac{r}{N}}T_{m+\frac{s2^l}{N}}\psi\|\\
&=\|x-N^{-2}\sum_{r,s=0}^{N-1}\sum_{l,m=-M}^{M-1}  D^*_{-\frac{r}{N}}D^*_{-l}T^*_{-\frac{s2^l}{N}}T^*_{-m}\psi^*(x) D_{\frac{r}{N}}D_{l}T_{\frac{s2^l}{N}}T_{m}\psi\|\\
&=\|x-N^{-2}\sum_{r,s=0}^{N-1}\sum_{l,m=-M}^{M-1}  D^*_{-\frac{r}{N}}T^*_{-\frac{s}{N}}D^*_{-l}T^*_{-m}\psi^*(x) D_{\frac{r}{N}}T_{\frac{s}{N}}D_{l}T_{m}\psi\|\\
&=\|x-N^{-2}\sum_{r,s=0}^{N-1}\sum_{l,m=-M}^{M-1}  D^*_{-\frac{r}{N}}T^*_{-\frac{s}{N}}D^*_{-l}T^*_{-m}\psi^*(x) D_{\frac{r}{N}}T_{\frac{s}{N}}D_{l}T_{m}\psi\|\\
&\leq N^{-2}\sum_{r,s=0}^{N-1}\|x-\sum_{l,m=-M}^{M-1} D^*_{-\frac{r}{N}}T^*_{-\frac{s}{N}}D^*_{-l}T^*_{-m}\psi^*(x) D_{\frac{r}{N}}T_{\frac{s}{N}}D_{l}T_{m}\psi\|\\
&\leq N^{-2}\sum_{r,s=0}^{N-1}\|T_{-\frac{s}{N}}D_{-\frac{r}{N}}x-\sum_{l,m=-M}^{M-1} D^*_{-l}T^*_{-m}\psi^*(T_{-\frac{s}{N}}D_{-\frac{r}{N}}x) D_{l}T_{m}\psi\|\\
&< N^{-2} \sum_{r,s=0}^{N-1} \vp=\vp
\end{align*}

Thus, we have for all $N\in\N$ that 
\begin{equation}\label{E:part1}
\|x-\int_{-M}^M\int_{-M}^M\psi^{N*}_{a,b}(x)\psi^N_{a,b}\|dadb<\vp
\end{equation}

Let $x\in L_p$, $f\in L_{p'}$, and measurable $E\subset\R^2$.  For each $l,m\in\Z$ and $r,s\in\{0,1,...,N-1\}$ we let  $E_{l,r,m,s}=E\cap [l+\frac{r}{N},l+\frac{r+1}{N}]\times[m+\frac{s}{N},m+\frac{s+1}{N}]$.

We have that 
\begin{align*}
|\int_{(a,b)\in E}  &\psi^{N*}_{a,b}(x)f(\psi^N_{(a,b)})d(a,b)|=|\sum_{r,s=0}^{N-1}\sum_{l,m\in\Z} \psi^{*}_{l+\frac{r}{N},m+\frac{s 2^{l}}{N}}(x)f(\psi_{l+\frac rN,m+\frac{s 2^{l}}{N}})\lambda(E_{l,r,m,s})|\\
&\leq\|f\|\sum_{r,s=0}^{N-1}\|\sum_{l,m\in\Z} \psi^{*}_{l+\frac{r}{N},m+\frac{s 2^{l}}{N}}(x)\psi_{l+\frac rN,m+\frac{s 2^{l}}{N}}\lambda(E_{l,r,m,s})\|\\
&=\|f\|\sum_{r,s=0}^{N-1}\|\sum_{l,m\in\Z}  D^*_{-l-\frac{r}{N}}T^*_{-m-\frac{s2^l}{N}}\psi^*(x) D_{l+\frac{r}{N}}T_{m+\frac{s2^l}{N}}\psi \lambda(E_{l,r,m,s})\|\\
&=\|f\|\sum_{r,s=0}^{N-1}\|\sum_{l,m\in\Z-M} D^*_{-l}T^*_{-m}\psi^*(T_{-\frac{s}{N}}D_{-\frac{r}{N}}x) D_{l}T_{m}\psi \lambda(E_{l,r,m,s})\|\qquad\textrm{ as before.}\\
&\leq\|f\|CN^{-2}\sum_{r,s=0}^{N-1}\|\sum_{l,m\in\Z-M} D^*_{-l}T^*_{-m}\psi^*(T_{-\frac{s}{N}}D_{-\frac{r}{N}}x) D_{l}T_{m}\psi\|\qquad\textrm{ by $C$-unconditionality.}\\
&=\|f\|CN^{-2}\sum_{r,s=0}^{N-1}\|T_{-\frac{s}{N}}D_{-\frac{r}{N}}x\|\\
&=\|f\|CN^{-2}\sum_{r,s=0}^{N-1}\|x\|\qquad\textrm{ as $T_{-\frac{s}{N}}$ and $D_{-\frac{r}{N}}$ are isometries.} \\
&=\|f\|C\|x\|
 \end{align*}

Thus, the map $f\mapsto \int_{(a,b)\in E}  \psi^{N*}_{a,b}(x)f(\psi^N_{(a,b)})d(a,b)$ defines a bounded linear functional on $L_{p'}$ with norm at most $C\|x\|$.  Hence, as $L_p$ is reflexive, there exists $x_E\in L_p$ with $\|x_E\|\leq C\|x\|$ so that for all $f\in L_{p'}$ we have that 
$f(x_E)=\int_{(a,b)\in E}  \psi^{N*}_{a,b}(x)f(\psi^N_{(a,b)})d(a,b)$.  By \eqref{E:part1}, we have that $x_{\R^2}=x$.
Thus, $(\psi^N_{a,b},\psi^{N*}_{a,b})$ is a $C$-unconditional continuous Schauder frame of $L_p$.  By Lemma \ref{L:limit}, we have that $(\psi_{a,b},\psi_{a,b}^*)_{(a,b)\in\R^2}$ is a $C$-unconditional continuous Schauder frame of $L_p$.

\end{proof}

\section{Continuous Schauder frames for $\ell_p$ for $1<p<\infty$}\label{S:ex2}

Given a Banach space $X$, constructing a Schauder basis for $X$ can be done by finding a dense linearly independent sequence $(x_n)_{n=1}^\infty$ in $X$ so that the projection operators are uniformly bounded.  This is usually much easier done than proving that every vector in the space has a unique basis representation.  On the other hand, we don't have any other option than to show the reconstruction formula explicitly when proving that we have a continuous Schauder frame.  This makes constructing continuous Schauder frames often much harder than constructing Schauder bases.
  In this section we give a general procedure to construct a large class of non-trivial continuous Schauder frames for $\ell_p$ with $1<p<\infty$.  

The following lemma is very similar to Young's inequality for estimating the $L_p(\R)$ norm for convolutions of functions, and we prove it in a similar way.
\begin{lem}\label{T:Young}
Let $f\in L_1(\R)$ such that $\|\|f(t-n)\|_{\ell_1(\Z)}\|_{L_\infty(\R)}:=\sup_{t\in\R}\sum_{n\in\Z}|f(t-n)|<\infty$.  If $1< p<\infty$ and $(a_n)_{n\in\Z}\in \ell_p(\Z)$ then for $p'=p/(p-1)$, 
$$\int |\sum_{n\in\Z} f(t-n) a_n|^p \,dt\leq \|f\|_1 \|(a_n)\|_p^p \left(\sup_{t\in\R}\sum_{n\in\Z}|f(t-n)|\right)^{p/p'}  
$$
\end{lem}
\begin{proof}
\begin{align*}
\int |\sum_{n\in\Z} f(t-n) a_n|^p\,dt &\leq \int \left(\sum_{n\in\Z} |f(t-n) a_n|\right)^p\!\!\!dt\\
&=\int \left(\sum_{n\in\Z} |f(t-n)|^{1/p} |a_n||f(t-n)|^{1/p'}\right)^p\,dt\\
&\leq\int \left(\sum_{n\in\Z} |f(t-n)| |a_n|^p\right)\left(\sum_{n\in\Z}|f(t-n)|\right)^{p/p'}\!\!\!\!\!dt\quad\textrm{ by Holders}\\
&\leq \left(\int \sum_{n\in\Z} |f(t-n)| |a_n|^p\,dt\right)    \left(\sup_{t\in\R}\sum_{n\in\Z}|f(t-n)|\right)^{p/p'}  \\
&= \left(\sum_{n\in\Z} \int |f(t-n)| |a_n|^p\,dt\right)    \left(\sup_{t\in\R}\sum_{n\in\Z}|f(t-n)|\right)^{p/p'} \quad\textrm{ by Fubini}\\
&=  \|f\|_{L_1(\R)} \|(a_n)\|_{\ell_p(\Z)}^p     \left(\sup_{t\in\R}\sum_{n\in\Z}|f(t-n)|\right)^{p/p'} 
\end{align*}
\end{proof}

Using the above lemma, we are now able to construct a large class of continuous Schauder frames for $\ell_p$ for all $1<p<\infty$.

\begin{thm}\label{T:cframe}
Let $f:\R\rightarrow \R$ be a measurable function such that the following are satisfied.
\begin{enumerate}
\item $f\in L_1(\R)$.
\item $(f(\cdot-n))_{n\in\Z}$ is an ortho-normal sequence in $L_2(\R)$.
\item $\|\|f(t-n)\|_{\ell_1(\Z)}\|_{L_\infty(\R)}=\sup_{t\in\R}\sum_{n\in\Z}|f(t-n)|<\infty$
\end{enumerate}
Let $x_t=(f(t-n))_{n\in\Z}\in\ell_p(\Z)$ and $f_t=(f(t-n))_{n\in\Z}\in\ell_q(\Z)$.  Then $\psi:\R\rightarrow \ell_p(\Z)\times\ell_q(\Z)$ given by $\psi(t)=(x_t,f_t)$
 is a continuous Shauder frame of $\ell_p(\Z)$.   Furthermore, $\psi$ has suppression unconditionality constant $C_s=\|f\|_1sup_{t\in\R}\sum_{n\in\Z}|f(t-n)|$.
\end{thm}

\begin{proof}
Let $n\in\N$.  We have that 
$$\int f_t(e_n)e^*_n(x_t) dt=\int f(t-n) f(t-n)dt=\int |f(t)|^2dt=1
$$
and for all $m\in\Z$ with $n\neq m$ we have that
$$\int f_t(e_n)e^*_m(x_t) dt=\int f(t-n) f(t-m)dt=\int f(t) f(t-(m-n))dt=0.
$$
Thus, we just need to show that $\int f_t(x)x_t\,dt$ is Pettis integrable for all $x\in \ell_p$.
Let $E\subseteq\R$ be measurable, $x=(a_n)_{n\in\Z}\in \ell_p$, and $x^*=(b_n)_{n\in\Z}\in\ell_q$.  Let $g(t)=\sum_{n\in\Z}a_n 1_{[n,n+1)}$
\begin{align*}
|\int_E f_t(x)x^*(x_t)dt|&=|\int_E (\sum_{n\in\Z} f(t-n) a_n)(\sum_{m\in\Z}f(t-m)b_m)dt|\\
&\leq (\int |\sum_{n\in\Z} f(t-n) a_n|^p dt)^{1/p} (\int |\sum_{m\in\Z} f(t-m) b_m|^q dt)^{1/q}\quad\textrm{ by Holder's inequality}\\
&\leq \left(\|f\|_1^{1/p}\|(a_n)\|_p \left(\sup_{t\in\R}\sum_{n\in\Z}|f(t-n)|\right)^{1/q}\right)\left(\|f\|_1^{1/q}\|(b_n)\|_q \left(\sup_{t\in\R}\sum_{n\in\Z}|f(t-n)|\right)^{1/p}\right)\\
&\qquad\qquad\textrm{ by Lemma \ref{T:Young}}\\
&=C_s \|x\|_p \|x^*\|_q
\end{align*}

Thus, we have that $|\int_E f_t(x)x^*(x_t)dt|\leq C_s \|x\|_p \|x^*\|_q$ for all $x\in \ell_p$ and $x^*\in \ell_q$.  This gives that for each measurable $E$ and $x\in \ell_p$, the map $x^*\mapsto \int_E f_t(x)x^*(x_t)dt$ defines a bounded linear functional on $\ell_q$.  Hence there exists unique $x_E\in \ell_p$ such that $x^*(x_E)=\int_E f_t(x)x^*(x_t)dt$ for all $x^*\in \ell_q$.  Thus, $\int f_t(x)x_t\,dt$ is Pettis integrable for all $x\in \ell_p$.

\end{proof}

The following is an example  of using Theorem \ref{T:cframe}  to create a non-trivial continuous Schauder frame of $\ell_p$ for $1<p<\infty$.

\begin{example}
Let $(r_n)_{n\in\Z}$ be a sequence of different Rademacher functions on the interval $[0,1]$.  That is, $|r_n(x)|=1$ for all $x\in[0,1]$ and $\int_{0}^1 r_n(x) r_m(x) dx=0$  for all $n\neq m$.
Let $(a_n)_{n\in\Z}\in\ell_1$ such that $\|(a_n)_{n\in\Z}\|_2=1$.  Then $f=\sum_{n\in\Z} a_n T_n r_n$ satisfies the conditions of Theorem \ref{T:cframe} where $T_n$ is the operator which translates a function to the right by $n$.  The suppression unconditionality constant of the resulting continuous Schauder frame is $(\sum_{n\in\Z}|a_n|)^2$.
\end{example}

\section{Sampling continuous Schauder frames}\label{S:O}

Many important frames for Hilbert spaces arrise as samplings of continuous frames.  
In particular, wavelet frames, Gabor frames, and Fourier frames are all samplings of different continuous frames.  Futhermore, all the frames introduced by Daubechies, Grossmann, and Meyer \cite{DGM} in``Painless nonorthogonal expansions"are created by sampling different coherent states. Formally, if $(M,\Sigma,\mu)$ is a $\sigma$-finite measure space and $(x_t,f_t)_{t\in M}$ is a continuous frame of a Banach space $X$ and $(t_j)_{j=1}^\infty$ is a sequence in $M$ then $(x_{t_j},f_{t_j})_{j=1}^\infty$ is called a sampling of $(x_t,f_t)_{t\in M}$.    The discretization problem, posed  by Ali, Antoine, and Gazeau \cite{AAG2},  asks when a continuous frame of a Hilbert space can be sampled to obtain a frame.  
A solution for certain types of continuous frames was obtained by Fornasier and Rauhut using the theory of co-orbit spaces \cite{FoR} and a complete solution was recently given by Speegle and the second author \cite{FS} using the solution of the Kadison Singer problem by Marcus, Spielman, and Srivastava \cite{MSS}.  In particular, every bounded continuous frame on a Hilbert space may be sampled to obtain a discrete frame.  
\begin{problem}\label{P}
What are some Banach spaces where every bounded continuous Schauder frame may be sampled to obtain a discrete Schauder frame?   What are some Banach spaces where there exists a bounded continuous Schauder frame which cannot be sampled to obtain a discrete Schauder frame?
\end{problem}
Note that the discretization problem was solved for continuous Hilbert space frames, and thus Problem \ref{P} is even open for continuous Schauder frames for separable Hilbert spaces.

\end{document}